\documentclass{amsart}

\newtheorem{theorem}{Theorem}[section]

\newtheorem{lemma}[theorem]{Lemma}

\title[Another proof of Hill's criterion]{An alternative proof of Hill's criterion of freeness for abelian groups}
\author[J. E. Mac\'{\i}as-D\'{\i}az]{J. E. Mac\'{\i}as-D\'{\i}az}
\address{Departamento de Matem\'{a}ticas y F\'{\i}sica, Universidad Aut\'{o}noma de Aguascalientes, Avenida Universidad 940, Ciudad Universitaria, Aguascalientes, Ags. 20131, Mexico}
\email{jemacias@correo.uaa.mx}

\subjclass{Primary 20K20; Secondary 03E75, 20K25}
\keywords{Abelian group, freeness, Hill's criterion, purity, $G (\aleph _0)$-family.\\ {\em Palabras y frases clave.} Grupo abeliano, libertad, criterio de Hill, pureza, $G (\aleph _0)$-familia}
\date{\today}

\begin{document}

\begin{abstract}
In this note, we provide a different proof of Hill's criterion of freeness for abelian groups. Our proof hinges on the construction of suitable $G (\aleph _0)$-families of subgroups of the links in Hill's theorem and, ultimately, on the construction of such a family of pure subgroups of the group itself. \\

\noindent
{\scshape Resumen.} En este trabajo, se proporciona una nueva demostraci\'{o}n del criterio de Hill para grupos abelianos libres. La demostraci\'{o}n se basa en la construcci\'{o}n de una $G (\aleph _0)$-familia de subgrupos en los eslabones del teorema de Hill y, prioritariamente, en la construcci\'{o}n de una familia tal de subgrupos puros.
\end{abstract}

\maketitle

\section{Introduction\label{Sec:Intro}}

In $1934$, Lev Pontryagin proved that a countable, torsion-free abelian group is free if and only if every finite rank, pure subgroup is free \cite {Pontryagin}. Equivalently, every properly ascending chain of subgroups of the same finite rank is finite. From the proof of this criterion, it follows that a torsion-free abelian group $G$ is free if there exists an ascending chain 
\begin{equation}
0 = G _0 < G _1 < \dots < G _n < \dots \quad (n < \omega), \label{Eq:PontChain}
\end{equation}
consisting of pure subgroups of $G$ whose union is equal to $G$, such that every $G _n$ is free and countable. Here, a subgroup $H$ of the abelian group $G$ is \emph {pure} if solubility in $G$ of every equation of the form $n x = h \in H$, with $n \in \mathbb {Z}$, implies its solubility in $H$. Also, we say that $G$ is \emph {torsion-free} if $n = 0$ or $g = 0$, whenever $n \in \mathbb {Z}$ and $g \in G$ satisfy $n g = 0$.

Later, in $1970$, Hill established that, in order for an abelian group $G$ to be free, it is sufficient to prove that it is the union of a countable ascending chain \eqref {Eq:PontChain} consisting of free, pure subgroups \cite {Hill}. In other words, he proved the following theorem, establishing thus that the countability condition on the cardinality of the links of the chain was superfluous.

\begin{theorem}[Hill's criterion of freeness] \label{Thm:Hill}
A torsion-free abelian group $G$ is free if there exists a countable ascending chain
\begin{equation}
0 = G _0 < G _1 < \dots < G _n < \dots \quad (n < \omega)
\end{equation}
of subgroups of $G$, such that:
\begin{enumerate}
	\item[\rm (a)] every $G _n$ is free,
	\item[\rm (b)] every $G _n$ is a pure subgroup of $G$, and
	\item[\rm (c)] $G = \bigcup _{n < \omega} G _n$.
\end{enumerate}
\end{theorem}

In this note, we give a proof of Hill's criterion different from the one provided in \cite {Hill}. Our proof hinges on the construction of suitable classes of subgroups of the groups $G _n$ and, ultimately, on the construction of such a family consisting of pure subgroups of $G$. Section \ref {Sec:Proof} of this work contains the proof of Theorem \ref {Thm:Hill}, while Section \ref {Sec:Lemmas} presents some preliminary results.

\section{Preparatory lemmas\label{Sec:Lemmas}}

The following is a general result which will be used in the proof of Theorem \ref {Thm:Hill}. We refer to \cite {Jech} for definitions of the set-theoretical concepts.

\begin{lemma} 
\label{Lemma:Gen}
An abelian group $G$ is free if there exists a continuous, well-ordered, ascending chain
\begin{equation}
0 = A _0 < A _1 < \dots < A _\gamma < A _{\gamma + 1} < \dots \quad (\gamma < \tau) \label{Eq:GralChain}
\end{equation}
of subgroups of $G$, such that:
\begin{enumerate}
	\item[\rm (a)] every factor group $A _{\gamma + 1} / A _\gamma$ is free, and
	\item[\rm (b)] $G = \bigcup _{\gamma < \tau} A _\gamma$.
\end{enumerate}
\end{lemma}

\begin{proof}
The conclusion follows from the fact that $G$ is isomorphic to the direct sum of the factor groups $A _{\gamma + 1} / A _\gamma$, for $\gamma < \tau$.
\end{proof}

Recall that a \emph {$G (\aleph _0)$-family} of an abelian group $G$ is a collection $\mathcal {B}$ of subgroups of $G$, which satisfies the following properties:
\begin{enumerate}
	\item[(i)] $0$ and $G$ belong to $\mathcal {B}$,
	\item[(ii)] $\mathcal {B}$ is closed under unions of ascending chains, and
	\item[(iii)] for every $A _0 \in \mathcal {B}$ and every countable set $H \subseteq G$, there exists $A \in \mathcal {B}$ which contains both $A _0$ and $H$, such that $A / A _0$ is countable.
\end{enumerate}
Clearly, every abelian group has a $G (\aleph _0)$-family, namely, the collection of all its subgroups.

For the rest of this section, we will assume the hypotheses of Theorem \ref {Thm:Hill}. Under these circumstances, we fix a basis $X _n$ of $G _n$ for every $n < \omega$, and let $\mathcal {B} _n$ be the family of all subgroups of $G _n$ generated by subsets of $X _n$. Clearly, every member of $G _n$ is a direct summand of $G _n$ and, thus, a pure subgroup of $G$.

\begin{lemma}
The collection $\mathcal {B} ^\prime _n = \{ A \in \mathcal {B} _n : A + G _i \text{ is pure in } G \text {, for every } i < \omega \}$ is a $G (\aleph _0)$-family of pure subgroups of $G _n$, for every $n < \omega$.
\end{lemma}

\begin{proof}
All we need to check is that the countability condition is satisfied, since the other conditions of a $G (\aleph _0)$-family are obvious. So, let $A _0 \in \mathcal {B} ^\prime _n$, and let $H _0$ be a countable subset of $G _n$. Moreover, let $m < \omega$, and assume that we have already constructed a chain
\begin{equation}
A _0 < A _1 < \dots < A _m \label{Eq:Chain4}
\end{equation}
of groups in $\mathcal {B} _n$, such that:
\begin{enumerate}
	\item[1.] $H _0$ is contained in $A _1$,
	\item[2.] for every $j < m$, the group $A _{j + 1} / A _j$ is countable, and
	\item[3.] for every $j < m$ and every $i < \omega$, $(A _{j + 1} + G _i) / (A _0 + G _i)$ contains the purification of $(A _j + G _i) / (A _0 + G _i)$ in $G / (A _0 + G _i)$.
\end{enumerate}
To find the next member of \eqref {Eq:Chain4}, for every $i < \omega$, let $V _i \subseteq G _n$ be a complete set of representatives of the purification of $(A _m + G _i) / (A _0 + G _i)$ in $G / (A _0 + G _i)$. The sets $V _i$ are clearly countable, so that $H _{m + 1} = H _0 \cup \bigcup _{i < \omega} V _i$ is likewise countable. Therefore, there exists $A _{m + 1} \in \mathcal {B} _n$ containing both $A _m$ and $H _{m + 1}$, such that $A _{m + 1} / A _m$ is countable. Inductively, we construct a chain
\begin{equation}
A _0 < A _1 < \dots < A _m < \dots \quad (m < \omega) \label{Eq:ChainLemma2}
\end{equation}
of groups in $\mathcal {B} _n$, satisfying properties 1, 2 and 3 above, for every $m < \omega$. 

Evidently, the union $A$ of the links of \eqref {Eq:ChainLemma2} is a member of $\mathcal {B} _n$, $A / A _0$ is countable, and our construction guarantees that $(A + G _i) / (A _0 + G _i)$ is pure in $G / (A _0 + G _i)$. Thus, $A + G _i$ is pure in $G$ and, consequently, $A$ belongs to $\mathcal {B} ^\prime _n$.
\end{proof}

\begin{lemma}
The collection $\mathcal {B} = \{A < G : A \cap G _n \in \mathcal {B} ^\prime _n \text {, for every } n < \omega\}$ is a $G (\aleph _0)$-family of pure subgroups of $G$. \label{Lemma:3}
\end{lemma}

\begin{proof}
Again, only the countability condition merits attention; so, let $A _0 \in \mathcal {B}$, and let $H \subseteq G$ be countable. For every $k < \omega$, let $A _k ^0 = A _0 \cap G _k$. Moreover, let $n < \omega$, and assume that we have already constructed a finite ascending chain
\begin{equation}
A _0 < A _1 < \dots < A _n \label {Eq:ChainFinite}
\end{equation}
of subgroups of $G$, such that all factor groups $A _m / A _0$ are countable, for every $m \leq n$. Furthermore, suppose that each link $A _m$ in \eqref {Eq:ChainFinite} may be expressed as the union of a countable ascending chain
\begin{equation}
0 = A _0 ^m < A _1 ^m < \dots < A _k ^m < \dots \quad (k < \omega)
\end{equation}
of subgroups of $G$, such that:
\begin{enumerate}
\item[(a)] $A _k ^m \in \mathcal {B} _k ^\prime$, for every $k < \omega$ and every $m \leq n$,
\item[(b)] $A _k ^m$ is countable over $A _0 \cap G _k$, for every $k < \omega$ and every $m \leq n$, and
\item[(c)] $A _k ^m < A _m \cap G _k < A _k ^{m + 1}$, for every $k < \omega$ and $m + 1 \leq n$.
\end{enumerate}

For every $k < \omega$, the group $(A _n \cap G _k) / (A _0 \cap G _k)$ is countable, so we may fix a countable set of representatives $Y _k$ of $A _n \cap G _k$ modulo $A _0 \cap G _k$. Moreover, there exists $B _k \in \mathcal {B} ^\prime _k$ containing both $A _0 \cap G _k$ and $Y _k$, such that $B _k$ is countable over $A _0 \cap G _k$. Thus, any set of representatives $H _k$ of $B _k$ modulo $A _0 \cap G _k$ is countable.

In order to construct the next link in \eqref {Eq:ChainFinite}, assume that the groups in the ascending chain $0 = A _0 ^{n + 1} < A _1 ^{n + 1} < \dots < A _k ^{n + 1}$ have been built as needed, for some $k < \omega$, and let $Z _k \subseteq G _k$ be a set of representatives of $A _k ^{m + 1}$ modulo $A _0 \cap G _k$. Then, there exists $A _{k + 1} ^{n + 1} \in \mathcal {B} ^\prime _{k + 1}$ which contains $A _0 \cap G _{k + 1}$ and the countable set $Z _k \cup H _{k + 1} \cup (H \cap G _{k + 1})$, such that $A ^{n + 1} _{k + 1}$ is countable over $A _0 \cap G _{k + 1}$.

Clearly, the group $A = \bigcup _{n < \omega} A _n$ contains both $A _0$ and $H$, and is countable over $A _0$. Moreover, our construction guarantees that $A \cap G _k \in \mathcal {B} _k$, for every $k < \omega$. We conclude that $A \in \mathcal {B}$.
\end{proof}

Before we prove our next result, it is important to notice that $A + G _n$ is a pure subgroup of $G$, for every $A \in \mathcal {B}$ and every $n < \omega$. Indeed, that $(A + G _n) \cap G _{n + 1}$ is pure in $G$ follows from the fact that $A \cap G _{n + 1} \in \mathcal {B} ^\prime _{n + 1}$. Next, assume that $(A + G _n) \cap G _k$ is pure in $G$, for some $k > n$. It is easy to check that 
\begin{equation}
\frac {(A + G _k) \cap G _{k + 1}} {(A + G _n) \cap G _{k + 1}} \cong \frac {G _k} {(A + G _n) \cap G _k},
\end{equation}
whence it follows that $(A + G _n) \cap G _{k + 1}$ is pure in $G$. The claim is readily established after noticing that $A + G _n = \bigcup _{k < \omega} (A + G _n) \cap G _k$.

\begin{lemma}
\label{Lemma:4}
For every $A \in \mathcal {B}$, finite rank, pure subgroups of $G / A$ are free.
\end{lemma}

\begin{proof}
Let $A \in \mathcal {B}$, and let $D$ be a pure subgroup of $G$ containing $A$, such that $D / A$ is of finite rank. If $S = \{ d _1 , \dots , d _n \}$ is a complete set of representatives of a maximal independent system of $D$ modulo $A$, then there exists $k < \omega$ such that $S \subseteq G _k$. Then $A + (D \cap G _k) = D \cap (A + G _k)$ is a pure subgroup of $G$ containing $S$, which lies between $A$ and $D$. Therefore, $D = A + (D \cap G _k)$. The fact that $A \cap G _k \in \mathcal {B} ^\prime _k$ implies that $A \cap G _k$ is a summand of $G _k$. Therefore, there exists a finite rank, free group $B$, such that $D \cap G _k = (A \cap G _k) \oplus B$. Notice that
\begin{equation}
D = A + (D \cap G _k) = A + ((A \cap G _k) \oplus B) = A \oplus B,
\end{equation}
which implies that $D / A$ is free.
\end{proof}

\section{Proof of the main result\label{Sec:Proof}}

\begin{proof}[Proof of Theorem \ref {Thm:Hill}]
Let $\alpha$ be any nonzero ordinal, and let
\begin{equation}
0 = A _0 < A _1 < \dots < A _\gamma < A _{\gamma + 1} \dots \quad (\gamma < \alpha)
\end{equation}
be an ascending chain of subgroups in $\mathcal {B}$, such that all factor groups $A _{\gamma + 1} / A _\gamma$ are free. If $\alpha$ is a limit ordinal, then we let $A _\alpha = \bigcup _{\gamma < \alpha} A _\gamma$. Otherwise, there exists an ordinal $\beta$ such that $\alpha = \beta + 1$. In this case, if there exists $x \in G \setminus A _\beta$, we let $A _{\beta + 1} \in \mathcal {B}$ contain both $x$ and $A _\beta$, such that $A _{\beta + 1} / A _\beta$ be countable. Lemma \ref {Lemma:4} implies now that finite rank, pure subgroups of $A _{\beta + 1} / A _\beta$ are free. Consequently, $A _{\beta + 1} / A _\beta$ is free by Pontryagin's criterion. 

Using transfinite induction, we construct a continuous, well-ordered, ascending chain \eqref {Eq:GralChain} of subgroups of $G$ satisfying properties (a) and (b) of Lemma \ref {Lemma:Gen}. We conclude that $G$ is free.
\end{proof}

\end{document}